\documentclass[12pt]{article}

\usepackage{amsmath,amssymb,mathtools}
\usepackage{microtype}

\allowdisplaybreaks

\newtheorem{theorem}{Theorem}[section]
\newtheorem{corollary}[theorem]{Corollary}
\newtheorem{lemma}[theorem]{Lemma}
\newtheorem{proposition}[theorem]{Proposition}
\newtheorem{remark}[theorem]{Remark}

\newcommand{\one}{\mathbf{1}}
\newcommand{\calB}{\mathcal{B}}
\newcommand{\R}{\mathbb{R}}
\newcommand{\eps}{\varepsilon}
\newcommand{\tr}{\operatorname{tr}}
\newcommand{\spanop}{\operatorname{span}}

\newenvironment{pf}[1][Proof]{\noindent\textbf{#1.} }{\hfill\rule{1mm}{2mm} \par\vspace{0.6\baselineskip}}

\begin{document}

\title{Extremal graphs for a conjecture on the square energy of graphs}
\author{Fu-Tao Hu, Ya-Yang Liu, Yi Wang\thanks{E-mail address: wangy@ahu.edu.cn} \\
{\small Center for Pure Mathematics, School of Mathematical Sciences, Anhui University,}\\
{\small Hefei, 230601, P.R. China}
}

\date{}
\maketitle

\begin{abstract}
For a graph $G$, let $s^+(G)$ and $s^-(G)$ denote the sums of the squares
of its positive and negative adjacency eigenvalues.  We determine all
equality cases in the conjecture of Elphick, Farber,
Goldberg, and Wocjan that every connected graph $G$ on $n$ vertices satisfies
\[
  \min \{s^+(G),s^-(G)\}\ge n-1.
\]
Namely, equality for $s^+$ holds exactly for trees,
whereas equality for $s^-$ holds exactly for trees and complete graphs.  
The proof combines the $P_3$-removal lemma in the no-cut-vertex case with a
detailed equality analysis of the underlying doubly nonnegative matrix
inequality.  Every block is forced to be complete, and a
minimal-counterexample argument gives an exact rank-one decomposition of
the folded matrix $M^c$.  The resulting non-edge vanishings, together with
$AX=XA$, rule out an interface between a bridge and a nontrivial block.

 
\end{abstract}

\section{Introduction}

Let $G$ be a finite simple graph with adjacency matrix $A=A(G)$ and
adjacency eigenvalues $\lambda_1,\ldots,\lambda_n$.  The energy of $G$,
introduced in connection with molecular orbital theory, is the
$\ell_1$-norm of its adjacency spectrum,
\[
  \mathcal E(G)=\sum_{i=1}^n |\lambda_i|;
\]
see the monograph of Li, Shi, and Gutman \cite{LiShiGutman}.  Positive and
negative square energies are the corresponding quadratic quantities on the
two sides of the spectrum.  More generally, positive and negative
$p$-energies have become useful spectral invariants in problems involving
colorings, homomorphisms, graph operations, and extremal eigenvalue
inequalities; see, for example,
\cite{WocjanElphick,AndoLin,GuoSpiro,CoutinhoSpier,
CoutinhoSpierZhang,AkbariVertex,ElphickTangZhang,TangLiuWang}.

We write $N_G(v)$ for the neighborhood of a vertex $v$ and define
\[
  s^+(G)=\sum_{\lambda_i>0}\lambda_i^2,
  \qquad
  s^-(G)=\sum_{\lambda_i<0}\lambda_i^2.
\]
These quantities arose naturally in spectral lower bounds for chromatic
parameters.  Wocjan and Elphick \cite{WocjanElphick} proposed a bound using
all adjacency eigenvalues, and Ando and Lin \cite{AndoLin} proved, for every
graph with at least one edge, the ratio inequality
\[
  1+\max\left\{\frac{s^+(G)}{s^-(G)},
                    \frac{s^-(G)}{s^+(G)}\right\}\le \chi(G).
\]
Guo and Spiro \cite{GuoSpiro} replaced the chromatic number by the
fractional chromatic number, while Coutinho and Spier
\cite{CoutinhoSpier} obtained a vector-chromatic formulation.  The
semidefinite and conic mechanisms behind such estimates were developed
further by Coutinho, Spier, and Zhang \cite{CoutinhoSpierZhang}; this line
of work is closely related to optimization over positive semidefinite and
doubly nonnegative cones \cite{BermanShaked,VandenbergheBoyd}.

Since $\tr A^2=2|E(G)|$, one always has
\begin{equation}\label{eq:trace-sum}
  s^+(G)+s^-(G)=2|E(G)|.
\end{equation}
In 2016, Elphick, Farber, Goldberg, and Wocjan
\cite{EFGW} conjectured that every connected graph on $n$ vertices satisfies
\begin{equation}\label{eq:EFGW-conjecture}
  \min\{s^+(G),s^-(G)\}\ge n-1.
\end{equation}
The problem was subsequently emphasized in the survey of Liu and Ning
\cite{LiuNing}.  By \eqref{eq:trace-sum}, the conjecture is equivalent to
\[
  \max\{s^+(G),s^-(G)\}\le 2|E(G)|-|V(G)|+1,
\]
which simultaneously controls the entire positive and negative parts of
the spectrum and strengthens Hong's classical estimate for the spectral
radius \cite{Hong}.  The bound is sharp at both extremal types of connected
graphs: every tree $T$ has
\[
  s^+(T)=s^-(T)=|E(T)|=|V(T)|-1,
\]
whereas $s^-(K_n)=n-1$.

A substantial literature developed around \eqref{eq:EFGW-conjecture}.  The
original paper \cite{EFGW} established the conjecture for several graph
classes, including bipartite and regular graphs.  Abiad, de Lima, Desai,
Guo, Hogben, and Madrid \cite{AbiadEtAl} obtained structural results and
verified further families.  Elphick and Linz \cite{ElphickLinz} studied the
pronounced asymmetry between the two square energies, while Elphick and
Aouchiche \cite{ElphickAouchiche} derived Nordhaus--Gaddum-type consequences
under the conjecture.  Zhang \cite{ZhangExtremal} introduced
semidefinite, superadditivity, and vertex-removal tools and proved
\[
  \min\{s^+(G),s^-(G)\}\ge n-\gamma(G)\ge \frac n2
\]
for graphs without isolated vertices, where $\gamma(G)$ is the domination
number.  Independently, Akbari, Kumar, Mohar, and Pragada
\cite{AkbariLinear} proved a general linear lower bound and later developed
vertex-partition methods for positive and negative $p$-energies
\cite{AkbariVertex}.  Further refinements for positive square energy,
including graphs with small domination number and a quantitative
$P_3$-removal lemma, were obtained by Akbari, Kumar, Mohar, Pragada, and
Zhang \cite{AKMPZ}.  Related progress includes results for unicyclic graphs
\cite{NingZeng}, monotonicity questions under edge addition
\cite{TangLiuWang}, and path-minimality phenomena for positive
$p$-energies \cite{LiuTangPath}.

Liu, Tang, and Zhang recently proved \eqref{eq:EFGW-conjecture} in full
\cite{LTZ}.  Their central idea is to replace the special matrices
$A_+\circ A_+$ and $A_-\circ A_-$ by the entire doubly nonnegative cone and
to prove a sharp graph-supported matrix inequality there.  The same
relaxation also yields a positive square-energy strengthening of Tur\'an's
theorem in a companion work \cite{LTZTuran}.  Once the non-strict lower
bound was settled, the natural remaining question was to determine all
equality cases.  Akbari, Kumar, Mohar, Pragada, and Zhang stated precisely
the following characterization as \cite[Conjecture~9.2]{AKMPZ}.

We prove that conjecture.

\begin{theorem}\label{thm:main}
Let $G$ be a connected graph on $n$ vertices.  Then
\[
  s^+(G)=n-1 \quad\Longleftrightarrow\quad G\text{ is a tree},
\]
and
\[
  s^-(G)=n-1 \quad\Longleftrightarrow\quad
  G\text{ is a tree or a complete graph}.
\]
\end{theorem}

\begin{corollary}\label{cor:strict}
If a connected graph $G$ is neither a tree nor a complete graph, then
\[
  \min \{s^+(G),s^-(G)\}>|V(G)|-1.
\]
\end{corollary}

Two different ingredients are needed.  First, the $P_3$-removal lemma of
Zhang \cite{ZhangExtremal} gives an immediate strict improvement over
$n-1$ whenever the original endpoint graph is noncomplete and has no cut
vertex.  Thus any remaining noncomplete equality graph has a nontrivial
block-cut tree.  Second, equality must be propagated through the doubly
nonnegative argument of Liu, Tang, and Zhang \cite{LTZ}.  This propagation
produces local matrices on the blocks, but those matrices need not be
Hadamard squares of spectral parts of the corresponding subgraphs.
Consequently, the $P_3$-removal lemma alone cannot determine the local block
structure.  We analyze equality in the no-cut-vertex part of the matrix
inequality and show that every block is complete and that its local matrix
is uniquely determined.

The remaining issue is compatibility across cut vertices.  We introduce
normalized equality pairs and retain the folding notation $M^c$ used in the
original argument.  A minimal-counterexample proof, obtained by peeling off
a leaf block, shows that
\[
  M^c=\sum_{B\in\mathcal B(G)}
  \left(\frac{|V(B)|-1}{|V(B)|}\right)^2
  \mathbf 1_B\mathbf 1_B^{\top}.
\]
This exact identity is stronger than merely knowing the local matrices.  It
implies, through blockwise zero-sum vectors, that specified non-edge entries
of the relevant spectral part vanish.  If a bridge meets a nontrivial complete
block, these vanishings force two different values for the same entry in
the commutation identity $AX=XA$, which is impossible.  The block-cut-tree
argument then gives Theorem~\ref{thm:main}.

The paper is organized as follows.  Section~\ref{sec:endpoint} records the
spectral reductions, the doubly nonnegative inequality, and the
$P_3$-removal treatment of graphs without cut vertices.
Section~\ref{sec:blocks} develops the equality structure and proves that all
blocks are complete.  Section~\ref{sec:leaf-exact} treats leaf blocks and
proves the exact decomposition of $M^c$.  The bridge--clique obstruction is
established in Section~\ref{sec:obstruction}, and
Section~\ref{sec:proof-main} completes the proof.

\section{Preliminaries and endpoint equality}
\label{sec:endpoint}

For real matrices $B,C$ of the same order, write
$\langle B,C\rangle=\tr(B^{\top}C)$ and
$\lVert B\rVert_F^2=\langle B,B\rangle$.  For a real symmetric matrix
$M$, write $M\succeq0$ when $M$ is positive semidefinite and $M\ge0$ when
it is entrywise nonnegative.  A matrix satisfying both conditions is called
\emph{doubly nonnegative}.  For each positive integer $r$, let $J_r$ denote
the $r\times r$ all-ones matrix; when its order is clear, we simply write
$J$.  For matrices $B,C$ of the same order, let $B\circ C$ denote their
Hadamard product.

The spectral decomposition of the adjacency matrix gives
\begin{equation}\label{eq:spectral-parts}
  A=A_+-A_-,
  \qquad
  A_+,A_-\succeq0,
  \qquad
  A_+A_-=0,
\end{equation}
where $A_+$ agrees with $A$ on its positive eigenspaces and $A_-$ agrees
with $-A$ on its negative eigenspaces.  Thus
\[
  s^+(G)=\tr(A_+^2),
  \qquad
  s^-(G)=\tr(A_-^2).
\]

For a connected graph $G$, put
\[
  q(G)=2|E(G)|-|V(G)|+1.
\]
For a doubly nonnegative matrix $M$ indexed by $V(G)$, define
\[
  S(G,M)=\sum_{uv\in E(G)}\sqrt{M_{uv}},
  \qquad
  T(M)=\one^{\top}M\one.
\]
We use the following theorem of Liu, Tang, and Zhang
\cite[Theorem~2.1]{LTZ}.

\begin{theorem}[Liu--Tang--Zhang]
\label{thm:dnn}
If $G$ is connected and $M$ is doubly nonnegative, then
\begin{equation}\label{eq:dnn}
  4S(G,M)^2\le q(G)T(M).
\end{equation}
\end{theorem}

The Schur product theorem \cite[Section~7.5]{HornJohnson} shows that
$A_+\circ A_+$ and $A_-\circ A_-$ are doubly nonnegative.  The proof in \cite{LTZ} applies
Theorem~\ref{thm:dnn} to these matrices and obtains
\begin{equation}\label{eq:upper-q}
  s^+(G),s^-(G)\le q(G),
\end{equation}
which is equivalent to the lower bounds
$s^+(G),s^-(G)\ge |V(G)|-1$ by \eqref{eq:trace-sum}.  We record all the
information forced by endpoint equality.

\begin{lemma}\label{lem:endpoint}
Let $G$ be connected on $n\ge2$ vertices.
\begin{enumerate}
\item If $s^+(G)=n-1$ and $X=A_-$, then
\[
  T(X\circ X)=q(G),
  \qquad
  4S(G,X\circ X)^2=q(G)T(X\circ X),
\]
and
\[
  X_{uv}\le0\qquad(uv\in E(G)).
\]
\item If $s^-(G)=n-1$ and $X=A_+$, then
\[
  T(X\circ X)=q(G),
  \qquad
  4S(G,X\circ X)^2=q(G)T(X\circ X),
\]
and
\[
  X_{uv}\ge0\qquad(uv\in E(G)).
\]
\end{enumerate}
\end{lemma}

\begin{pf}
Suppose first that $s^+(G)=n-1$.  By \eqref{eq:trace-sum},
$s^-(G)=q(G)$.  Using \eqref{eq:spectral-parts},
\[
  s^-(G)=-\langle A,A_-\rangle
  =-2\sum_{uv\in E(G)}(A_-)_{uv}
  \le 2\sum_{uv\in E(G)}\lvert(A_-)_{uv}\rvert.
\]
Consequently,
\[
  s^-(G)^2
  \le 4S(G,A_-\circ A_-)^2
  \le q(G)T(A_-\circ A_-).
\]
Moreover,
\[
  T(A_-\circ A_-)
  =\sum_{u,v}(A_-)_{uv}^2
  =\lVert A_-\rVert_F^2
  =s^-(G)=q(G).
\]
The first and last terms in the preceding chain are therefore both
$q(G)^2$, so equality holds throughout.  In particular,
\[
  \sum_{uv\in E(G)}\lvert(A_-)_{uv}\rvert
  =-\sum_{uv\in E(G)}(A_-)_{uv}.
\]
Since $|x|\ge -x$ term by term, $(A_-)_{uv}\le0$ on every edge.

The proof of the second assertion is identical.  If $s^-(G)=n-1$, then
$s^+(G)=q(G)$, and
\[
  s^+(G)=\langle A,A_+\rangle
  =2\sum_{uv\in E(G)}(A_+)_{uv}
  \le2\sum_{uv\in E(G)}\lvert(A_+)_{uv}\rvert.
\]
Equality throughout the resulting chain gives
$(A_+)_{uv}\ge0$ on every edge and all the remaining assertions.
\end{pf}

It will be convenient to treat the two endpoint cases simultaneously.  From
now on, under an endpoint hypothesis, write
\begin{equation}\label{eq:X-epsilon}
  (X,\eps)=
  \begin{cases}
    (A_-,-1),&s^+(G)=|V(G)|-1,\\
    (A_+, 1),&s^-(G)=|V(G)|-1.
  \end{cases}
\end{equation}
Thus, on every edge, $X_{uv}$ has sign $\eps$.

We shall also use the following vertex-removal result of Zhang.  It is
particularly effective when the graph under consideration has no cut
vertex.

\begin{lemma}[$P_3$-removal lemma {\cite[Theorem~1.10]{ZhangExtremal}}]
\label{lem:P3-removal}
Let $H$ be a graph, and suppose that $U\subseteq V(H)$ induces a copy of
$P_3$.  For each sign $\sigma\in\{+,-\}$, there exists a vertex $u\in U$
such that
\begin{equation}\label{eq:P3-removal}
  s^{\sigma}(H)>s^{\sigma}(H-u)+1.
\end{equation}
\end{lemma}

A quantitative refinement replaces the strict increment in
\eqref{eq:P3-removal} by $1+1/16$; see
\cite[Lemma~2.4]{AKMPZ}.  The original form above is already sufficient for
our equality problem.

\begin{lemma}[The endpoint case without cut vertices]
\label{lem:no-cut-endpoint}
Let $G$ be a connected graph on $n\ge2$ vertices with no cut vertex.  If
$s^+(G)=n-1$ or $s^-(G)=n-1$, then $G$ is complete.
\end{lemma}

\begin{pf}
Let $\sigma\in\{+,-\}$ be such that $s^\sigma(G)=n-1$.  Suppose that $G$
is not complete.  Since $G$ is connected, a shortest path
between two nonadjacent vertices contains three consecutive vertices that
induce a $P_3$.  Apply Lemma~\ref{lem:P3-removal} to the square energy under
consideration.  It gives a vertex $u$ of this induced $P_3$ such that
\[
  s^{\sigma}(G)>s^{\sigma}(G-u)+1.
\]
Because $G$ has no cut vertex, $G-u$ is connected.  The theorem of
Liu--Tang--Zhang gives
\[
  s^{\sigma}(G-u)\ge |V(G-u)|-1=n-2.
\]
Consequently $s^{\sigma}(G)>n-1$, contradicting the endpoint hypothesis.
Thus $G$ is complete.
\end{pf}

\begin{remark}\label{rem:P3-local-limitation}
Lemma~\ref{lem:no-cut-endpoint} completely settles the no-cut-vertex case
for the original endpoint graph.  It does not, however, replace the local
matrix rigidity needed below.  After splitting a normalized equality pair
at a cut vertex, the resulting local matrix is an arbitrary doubly
nonnegative matrix supported on the local graph; it need not be the
Hadamard square of a spectral part of that graph.  The $P_3$-removal lemma
therefore cannot be applied to those local pairs, and equality in the
no-cut-vertex part of Theorem~\ref{thm:dnn} must still be analyzed.
\end{remark}

\section{Equality forces complete blocks}
\label{sec:blocks}

We regard each bridge as a block $K_2$; every other block is a maximal
2-connected subgraph.  The set of blocks of a graph $H$ is denoted by
$\calB(H)$.  For $u\in V(H)$, let $e_u$ denote the corresponding standard
basis vector.  For a block $B$, let $\one_B\in\R^{V(H)}$ be the indicator
vector of $V(B)$.

A matrix $M$ indexed by $V(H)$ is said to be \emph{supported on the
diagonal and the edges of $H$} if
\[
  M_{uv}=0
  \qquad\text{whenever }u\ne v\text{ and }uv\notin E(H).
\]
We call $(H,M)$ a \emph{normalized equality pair} if $H$ is connected,
$|V(H)|\ge2$, $M$ is doubly nonnegative and supported on the diagonal and
the edges of $H$, and
\begin{equation}\label{eq:normalized-pair}
  T(M)=q(H),
  \qquad
  4S(H,M)^2=q(H)T(M)=q(H)^2.
\end{equation}

Following the folding step in \cite{LTZ}, for a doubly nonnegative matrix
$M$ indexed by $V(H)$, define its folding, denoted by $M^c$, by
\begin{equation}\label{eq:folding}
  M^c
  =M+\sum_{\{u,v\}\in\binom{V(H)}2\setminus E(H)}
  M_{uv}(e_u-e_v)(e_u-e_v)^{\top},
\end{equation}
where each unordered non-edge occurs once.  Every added summand is positive
semidefinite.  It cancels the corresponding off-diagonal non-edge entry,
preserves all edge entries, and annihilates $\one$.  Hence $M^c$ is
doubly nonnegative, is supported on the diagonal and the edges of $H$, and
\begin{equation}\label{eq:fold-preserve}
  S(H,M^c)=S(H,M),
  \qquad
  T(M^c)=T(M).
\end{equation}

The next lemma gives the precise equality propagation at a cut vertex.  The
fact that the residual vector may be assigned to either side will be crucial
in Section~\ref{sec:leaf-exact}.

\begin{lemma}[Equality propagation at a cut vertex]
\label{lem:cut-splitting}
Let $(H,M)$ be a normalized equality pair, let $v$ be a cut vertex of $H$,
and group the components of $H-v$ into two nonempty unions with vertex sets
$U$ and $W$.  Put
\[
  H_1=H[U\cup\{v\}],
  \qquad
  H_2=H[W\cup\{v\}].
\]
Take a Gram representation $M_{xy}=\langle z_x,z_y\rangle$, and write
\begin{equation}\label{eq:gram-cut}
  z_v=p_U+p_W+p_0,
\end{equation}
where $p_U$ and $p_W$ are the orthogonal projections of $z_v$ onto
\[
  \spanop\{z_u:u\in U\}
  \qquad\text{and}\qquad
  \spanop\{z_w:w\in W\},
\]
respectively, and $p_0$ is orthogonal to both spans.  Assign $p_0$ to either one of the
two sides, and let $M_1,M_2$ be the corresponding Gram matrices on $H_1$
and $H_2$.  Then both $(H_1,M_1)$ and $(H_2,M_2)$ are normalized equality
pairs.
\end{lemma}

\begin{pf}
Because $M$ is supported on the diagonal and the edges of $H$, no entry
between $U$ and $W$ is nonzero.  Hence the two spans in
\eqref{eq:gram-cut} are orthogonal.  If, for example, $p_0$ is assigned to
the first side, then $M_1$ is the Gram matrix of the vectors $(z_u)_{u\in U}$
together with $p_U+p_0$ at $v$, while $M_2$ is the Gram matrix of
$(z_w)_{w\in W}$ together with $p_W$ at $v$.  The off-diagonal entries
involving $v$ agree with the corresponding entries of $M$, and all other
off-diagonal entries are inherited from $M$.  Thus $M_1,M_2$ are doubly
nonnegative and supported on the corresponding graphs.  The same is true if
$p_0$ is assigned to the second side.

For either assignment, the edge sets partition $E(H)$ and the Gram sums
split orthogonally, so
\begin{equation}\label{eq:split-ST}
  S(H,M)=S(H_1,M_1)+S(H_2,M_2),
  \qquad
  T(M)=T(M_1)+T(M_2).
\end{equation}
Also,
\[
  q(H_1)+q(H_2)=q(H).
\]
Set $q_i=q(H_i)$ and $T_i=T(M_i)$.  Since each $H_i$ is connected and has
at least two vertices, $q_i>0$.  Theorem~\ref{thm:dnn} and scalar
Cauchy--Schwarz give
\begin{equation}\label{eq:cut-chain}
  2S(H,M)
  \le \sqrt{q_1T_1}+\sqrt{q_2T_2}
  \le \sqrt{(q_1+q_2)(T_1+T_2)}
  =\sqrt{q(H)T(M)}.
\end{equation}
By \eqref{eq:normalized-pair}, both ends of \eqref{eq:cut-chain} equal
$q(H)$.  Hence equality holds at every step.  The first equality implies
that equality holds in Theorem~\ref{thm:dnn} for each $(H_i,M_i)$.
Equality in scalar Cauchy--Schwarz gives
\[
  \frac{T_1}{q_1}=\frac{T_2}{q_2}=:\rho.
\]
Using \eqref{eq:split-ST},
\[
  q(H)=T(M)=T_1+T_2=\rho(q_1+q_2)=\rho q(H),
\]
so $\rho=1$.  Therefore $T(M_i)=q(H_i)$ for $i=1,2$, and both local pairs
are normalized equality pairs.
\end{pf}

For completeness, we isolate the two estimates from the no-cut-vertex part
of the proof of Theorem~\ref{thm:dnn} that will be used below.

\begin{lemma}[No-cut-vertex estimates]
\label{lem:no-cut-estimates}
Let $H$ be a connected graph on $r\ge3$ vertices such that $H-v$ is
connected for every $v\in V(H)$.  Let $M$ be doubly nonnegative and
supported on the diagonal and the edges of $H$.  Write
\[
  m=|E(H)|,
  \qquad q=q(H),
  \qquad d_0=\tr M,
  \qquad w=\sum_{uv\in E(H)}M_{uv}.
\]
Then
\begin{align}
  4S(H,M)^2&\le 4mw,\label{eq:flat-estimate}\\
  4S(H,M)^2&\le(q-1)T(M)+\frac{q-1}{r-2}d_0.
  \label{eq:averaged-estimate}
\end{align}
\end{lemma}

\begin{pf}
The flat estimate \eqref{eq:flat-estimate} is Cauchy--Schwarz over the $m$
edges.

For \eqref{eq:averaged-estimate}, write $M-v$ for the principal submatrix
obtained by deleting the row and column indexed by $v$, and let
\[
  \sigma_v=\sum_{u\sim v}\sqrt{M_{uv}}.
\]
Since $H-v$ is connected, Theorem~\ref{thm:dnn} applied to the principal
submatrix $M-v$ gives
\[
  2\bigl(S(H,M)-\sigma_v\bigr)
  \le\sqrt{q(H-v)T(M-v)}.
\]
Summing over $v$ and using $\sum_v\sigma_v=2S(H,M)$ yields
\[
  2(r-2)S(H,M)
  \le\sum_v\sqrt{q(H-v)T(M-v)}.
\]
A further application of Cauchy--Schwarz gives
\begin{equation}\label{eq:delete-CS}
  2(r-2)S(H,M)
  \le
  \sqrt{\left(\sum_v q(H-v)\right)
             \left(\sum_v T(M-v)\right)}.
\end{equation}
Since $q(H-v)=q+1-2d_H(v)$ and $\sum_vd_H(v)=2m$,
\[
  \sum_vq(H-v)=(r-2)(q-1).
\]
Every diagonal entry of $M$ survives in $r-1$ of the principal submatrices
$M-v$, whereas every off-diagonal edge entry survives in $r-2$ of them.
Therefore
\[
  \sum_vT(M-v)=(r-2)T(M)+d_0.
\]
Substituting these identities into \eqref{eq:delete-CS}, squaring, and
dividing by $(r-2)^2$ proves \eqref{eq:averaged-estimate}.
\end{pf}

\begin{lemma}[Rigidity of a one-block equality pair]
\label{lem:one-block}
Let $(H,M)$ be a normalized equality pair, and suppose that $H$ has exactly
one block.  Then $H=K_r$ for some $r\ge2$, and
\begin{equation}\label{eq:one-block-matrix}
  M=\alpha_H^2J_r,
  \qquad
  \alpha_H=\frac{r-1}{r}.
\end{equation}
\end{lemma}

\begin{pf}
If $H=K_2$, write
\[
  M=\begin{pmatrix}a&c\\c&b\end{pmatrix}.
\]
Here $q(H)=T(M)=1$, and equality in \eqref{eq:dnn} gives $4c=1$; hence
$c=1/4$ and $a+b=1/2$.  Positive semidefiniteness gives
$ab\ge c^2=1/16$, while $ab\le(a+b)^2/4=1/16$.  Thus
$a=b=c=1/4$, which is \eqref{eq:one-block-matrix} for $r=2$.

Now assume $r=|V(H)|\ge3$.  Since $H$ has one block, $H-v$ is connected
for every vertex $v$.  Put
\[
  m=|E(H)|,
  \qquad q=q(H),
  \qquad d_0=\tr M,
  \qquad w=\sum_{uv\in E(H)}M_{uv}.
\]
We first show that the second case in the proof of
Theorem~\ref{thm:dnn} is necessarily strict.  Suppose
\begin{equation}\label{eq:second-case}
  2(r-1)w>qd_0.
\end{equation}
Set $\beta=m-r+1$.  Then $q=2\beta+r-1$ and, because $H$ is simple,
\[
  2\beta\le(r-1)(r-2).
\]
Consequently,
\[
  q(r-2)-2\beta(r-1)
  =(r-1)(r-2)-2\beta\ge0.
\]
Multiplying \eqref{eq:second-case} by $r-2$ gives
\[
  2\beta(r-1)d_0
  \le q(r-2)d_0
  <2(r-1)(r-2)w,
\]
and hence
\[
  \beta d_0<(r-2)w.
\]
Since $q-1=(r-2)+2\beta$ and $T(M)=d_0+2w$, this is equivalent to
\[
  \frac{q-1}{r-2}d_0<T(M).
\]
The averaged estimate \eqref{eq:averaged-estimate} would then give
\[
  4S(H,M)^2<qT(M),
\]
contrary to \eqref{eq:normalized-pair}.  Therefore
\begin{equation}\label{eq:first-case}
  2(r-1)w\le qd_0.
\end{equation}

Using $4m=2q+2(r-1)$, the flat estimate and
\eqref{eq:first-case} give
\[
  q^2=4S(H,M)^2
  \le4mw
  =2qw+2(r-1)w
  \le q(2w+d_0)
  =qT(M)=q^2.
\]
Thus equality holds throughout.  Equality in Cauchy--Schwarz in
\eqref{eq:flat-estimate} implies that all edge entries of $M$ have a common
value $c$.  Since $S(H,M)=m\sqrt c$, we obtain
\begin{equation}\label{eq:common-edge}
  c=\left(\frac{q}{2m}\right)^2.
\end{equation}
Set $\alpha=q/(2m)$.  From $T(M)=q$ and \eqref{eq:common-edge},
\begin{equation}\label{eq:diagonal-mass}
  d_0=q-2m\alpha^2
  =\frac{q(r-1)}{2m}
  =\alpha(r-1).
\end{equation}
Substituting equality, $T(M)=q$, and \eqref{eq:diagonal-mass} into
\eqref{eq:averaged-estimate} yields
\[
  q^2
  \le q(q-1)
     +\frac{q-1}{r-2}\frac{q(r-1)}{2m}.
\]
After division by $q>0$, this is equivalent to
\[
  2m(r-2)\le(q-1)(r-1).
\]
Since $2m=q+r-1$, the last inequality is equivalent to
\[
  q\ge(r-1)^2.
\]
On the other hand, simplicity gives
\[
  q=2m-r+1\le r(r-1)-r+1=(r-1)^2.
\]
Hence equality holds, $m=\binom r2$, and $H=K_r$.  Moreover,
\[
  q=(r-1)^2,
  \qquad
  \alpha=\frac{q}{2m}=\frac{r-1}{r}=\alpha_H.
\]

It remains to determine the diagonal of $M$.  Every off-diagonal entry is
$c=\alpha_H^2$, and \eqref{eq:diagonal-mass} gives
\[
  \sum_{i=1}^rM_{ii}=r\alpha_H^2=rc.
\]
Every $2\times2$ principal minor is positive semidefinite, so
$M_{ii}M_{jj}\ge c^2$ for $i\ne j$.  Multiplying over all unordered pairs,
\[
  \left(\prod_{i=1}^rM_{ii}\right)^{r-1}
  \ge c^{2\binom r2}=c^{r(r-1)},
\]
and therefore $\prod_iM_{ii}\ge c^r$.  By the arithmetic--geometric mean
inequality,
\[
  rc=\sum_{i=1}^rM_{ii}
  \ge r\left(\prod_{i=1}^rM_{ii}\right)^{1/r}
  \ge rc.
\]
Equality holds throughout, so $M_{ii}=c$ for every $i$.  This proves
\eqref{eq:one-block-matrix}.
\end{pf}

\begin{corollary}\label{cor:complete-blocks}
Let $(H,M)$ be a normalized equality pair.  Then every block $B$ of $H$ is
complete.  More precisely, repeatedly applying the cut-vertex splitting of
Lemma~\ref{lem:cut-splitting} until the block $B$ is isolated produces the
unique local matrix
\begin{equation}\label{eq:local-block-matrix}
  M_B=\alpha_B^2J_{|V(B)|},
  \qquad
  \alpha_B=\frac{|V(B)|-1}{|V(B)|}.
\end{equation}
\end{corollary}

\begin{pf}
Lemma~\ref{lem:cut-splitting} shows at every step that the local pairs are
again normalized equality pairs.  Once $B$ is isolated, the resulting graph
has exactly one block.  Lemma~\ref{lem:one-block} gives both the completeness
of $B$ and \eqref{eq:local-block-matrix}.
\end{pf}

Under either endpoint equality, let $X$ and $\eps$ be as in
\eqref{eq:X-epsilon}, and set
\begin{equation}\label{eq:endpoint-M}
  M=X\circ X.
\end{equation}
Using the folding notation in \eqref{eq:folding},
\begin{equation}\label{eq:endpoint-folding}
  M^c=X\circ X+
  \sum_{\{u,v\}\in\binom{V(G)}2\setminus E(G)}
  X_{uv}^2(e_u-e_v)(e_u-e_v)^{\top}.
\end{equation}
By Lemma~\ref{lem:endpoint} and \eqref{eq:fold-preserve}, $(G,M^c)$ is a
normalized equality pair.  We therefore obtain the following consequences.

\begin{proposition}\label{prop:endpoint-blocks}
Assume either endpoint equality in Lemma~\ref{lem:endpoint}.  Then every
block of $G$ is complete.  If an edge $uv$ belongs to a block $B=K_r$, then
\begin{equation}\label{eq:X-on-edge}
  X_{uv}=\eps\alpha_B
  =\eps\frac{r-1}{r}.
\end{equation}
In particular, if $uv$ is a bridge, then
\begin{equation}\label{eq:bridge-parts}
  (A_+)_{uv}=\frac12,
  \qquad
  (A_-)_{uv}=-\frac12.
\end{equation}
\end{proposition}

\begin{pf}
Corollary~\ref{cor:complete-blocks} applied to $(G,M^c)$ shows that every block
is complete.  Folding does not change edge entries, and the cut-vertex splittings in
Lemma~\ref{lem:cut-splitting} preserve them.  Thus, if $uv\in E(B)$, then
\[
  X_{uv}^2=(M^c)_{uv}=\alpha_B^2.
\]
The edge-sign conclusion in Lemma~\ref{lem:endpoint} gives
\eqref{eq:X-on-edge}.  For a bridge, $\alpha_B=1/2$.  In either endpoint
case, \eqref{eq:X-on-edge} gives one of the two equalities in
\eqref{eq:bridge-parts}; the other follows from
\[
  1=A_{uv}=(A_+)_{uv}-(A_-)_{uv}.
\]
\end{pf}

\section{Leaf blocks and the exact block decomposition}
\label{sec:leaf-exact}

When $G$ has more than one block, a \emph{leaf block} is a block whose
node is a leaf of the block-cut tree.  Such a block contains exactly one
cut vertex of $G$.  When $G$ itself has one block, we also regard that block
as a leaf block.

\begin{lemma}\label{lem:positive-leaf}
If $G$ is connected and has a nontrivial leaf block, then
\[
  s^+(G)>|V(G)|-1.
\]
\end{lemma}

\begin{pf}
By the non-strict bound of Liu, Tang, and Zhang \cite{LTZ}, it is enough to
exclude equality.  Suppose that $s^+(G)=n-1$.  By
Proposition~\ref{prop:endpoint-blocks}, a nontrivial leaf block is $B=K_t$ for
some $t\ge3$.  If $B=G$, then
\[
  s^+(G)=(t-1)^2>t-1=n-1,
\]
a contradiction.  Thus $B$ is a proper leaf block.  Let $v$ be its unique
cut vertex and put $L=V(B)\setminus\{v\}$.  No vertex of $L$ has a
neighbor outside $B$, since otherwise it would lie in a second block and
would be another cut vertex of the leaf block.  Hence, for distinct $x,y\in L$, the vertices $x$
and $y$ are adjacent twins, so
\[
  A(e_x-e_y)=-(e_x-e_y).
\]
Hence
\[
  A_-(e_x-e_y)=e_x-e_y.
\]
Set $\alpha=(t-1)/t$.  By \eqref{eq:X-on-edge},
$(A_-)_{xy}=-\alpha$.  Taking the $x$-coordinate in the preceding
eigenvector identity gives
\[
  (A_-)_{xx}-(A_-)_{xy}=1,
\]
so $(A_-)_{xx}=1/t$.  Therefore the principal submatrix $A_-[L]$ has
diagonal entries $1/t$ and off-diagonal entries $-(t-1)/t$.  Its eigenvalue
in the $\one_L$-direction is
\[
  \frac1t-(t-2)\frac{t-1}{t}
  =\frac{1-(t-1)(t-2)}{t}<0,
\]
contrary to $A_-\succeq0$.
\end{pf}

\begin{lemma}\label{lem:negative-leaf}
Suppose $s^-(G)=|V(G)|-1$.  If $G$ has a nontrivial leaf block $B$, then
$G=B$ is complete.
\end{lemma}

\begin{pf}
By Proposition~\ref{prop:endpoint-blocks}, write $B=K_t$ with $t\ge3$.
If $B=G$, there is nothing to prove.  Suppose that $B$ is a proper leaf
block.  Let $v$ be its cut vertex and put $L=V(B)\setminus\{v\}$.  As in
the preceding proof, no vertex of $L$ has a neighbor outside $B$.  Set
\[
  P=A_+,
  \qquad
  \alpha=\frac{t-1}{t}.
\]
For distinct $\ell,\ell'\in L$,
\[
  A(e_\ell-e_{\ell'})=-(e_\ell-e_{\ell'}),
\]
so $P(e_\ell-e_{\ell'})=0$.  Together with
\eqref{eq:X-on-edge}, this gives
\begin{equation}\label{eq:P-on-B}
  P_{\ell\ell'}=P_{\ell v}=P_{\ell\ell}=\alpha
  \qquad
  (\ell,\ell'\in L,\ \ell\ne\ell').
\end{equation}
Indeed, the first two equalities are the edge values, and the last follows
by taking the $\ell$-coordinate in
$P(e_\ell-e_{\ell'})=0$.

Because $P$ is the positive spectral part of $A$,
\begin{equation}\label{eq:P2AP}
  P^2=AP.
\end{equation}
At the entry $(\ell,\ell)$, \eqref{eq:P-on-B} gives
\[
  (AP)_{\ell\ell}
  =\sum_{x\sim\ell}P_{x\ell}
  =(t-1)\alpha=t\alpha^2.
\]
The $t$ entries of row $\ell$ indexed by $B$ already contribute
$t\alpha^2$ to
\[
  (P^2)_{\ell\ell}=\sum_xP_{\ell x}^2.
\]
All remaining summands are nonnegative, and therefore
\begin{equation}\label{eq:P-leaf-zero}
  P_{\ell w}=0
  \qquad(\ell\in L,\ w\notin B).
\end{equation}

Fix $w\notin B$.  Expanding \eqref{eq:P2AP} at $(\ell,w)$ and using
\eqref{eq:P-on-B}--\eqref{eq:P-leaf-zero},
\[
  (P^2)_{\ell w}
  =\sum_xP_{\ell x}P_{xw}
  =\alpha P_{vw},
\]
whereas, since $N_G(\ell)=B\setminus\{\ell\}$,
\[
  (AP)_{\ell w}
  =\sum_{x\sim\ell}P_{xw}
  =P_{vw}.
\]
Thus $\alpha P_{vw}=P_{vw}$.  Since $\alpha<1$, we have
$P_{vw}=0$ for every $w\notin B$.  But a proper leaf block has a neighbor
$w\notin B$ of its cut vertex $v$, and \eqref{eq:X-on-edge} gives
$P_{vw}>0$ on that edge, a contradiction.  Hence $G=B$.
\end{pf}

We next strengthen the local block information to a global matrix identity.
The proof by a minimal counterexample makes explicit both the disappearance
of the residual Gram component and the diagonal entry at each cut vertex.

\begin{proposition}[Exact decomposition of normalized equality pairs]
\label{prop:exact-decomposition}
Let $(H,M)$ be a normalized equality pair.  Then
\begin{equation}\label{eq:general-exact-decomposition}
  M=\sum_{B\in\calB(H)}
  \alpha_B^2\one_B\one_B^{\top},
  \qquad
  \alpha_B=\frac{|V(B)|-1}{|V(B)|}.
\end{equation}
\end{proposition}

\begin{pf}
Suppose that the assertion is false, and choose a counterexample $(H,M)$
with the minimum possible number of blocks.  If $H$ has only one block,
Lemma~\ref{lem:one-block} gives \eqref{eq:general-exact-decomposition}, a
contradiction.  Thus $H$ has at least two blocks.

Choose a leaf block $B$, let $v$ be its unique cut vertex, and set
\[
  H'=H-\bigl(V(B)\setminus\{v\}\bigr).
\]
Because $B$ is a leaf block, $V(B)\setminus\{v\}$ is a component of
$H-v$.  Consequently, $H'$ is connected, and its blocks are precisely the
blocks of $H$ other than $B$.  Put
\[
  U=V(B)\setminus\{v\},
  \qquad
  W=V(H')\setminus\{v\}.
\]
Take a Gram representation $M_{xy}=\langle z_x,z_y\rangle$.  Since $M$ is
supported on the diagonal and the edges of $H$, the spaces
\[
  \mathcal U=\spanop\{z_u:u\in U\},
  \qquad
  \mathcal W=\spanop\{z_w:w\in W\}
\]
are orthogonal.  Write
\begin{equation}\label{eq:leaf-gram}
  z_v=p_B+p'+p_0,
  \qquad
  p_B\in\mathcal U,
  \quad p'\in\mathcal W,
  \quad p_0\perp(\mathcal U+\mathcal W).
\end{equation}

There are two valid cut-vertex splittings: one assigns $p_0$ to the $B$-side,
and the other assigns $p_0$ to the $H'$-side.  By
Lemma~\ref{lem:cut-splitting}, the matrix on the $B$-side has mass $q(B)$
in either splitting.  The two masses are
\[
  \left\lVert\sum_{u\in U}z_u+p_B+p_0\right\rVert^2
  \quad\text{and}\quad
  \left\lVert\sum_{u\in U}z_u+p_B\right\rVert^2.
\]
Because $p_0\perp\mathcal U$, their difference is $\lVert p_0\rVert^2$.
Hence $p_0=0$.

Let $M_B$ and $M'$ be the local Gram matrices obtained from the vectors on
$B$ and $H'$, with $p_B$ and $p'$ placed at $v$, respectively.  Both
$(B,M_B)$ and $(H',M')$ are normalized equality pairs by
Lemma~\ref{lem:cut-splitting}.  Extend $M_B$ and $M'$ by zero to matrices on
$V(H)$ and denote the extensions by $\widetilde M_B$ and $\widetilde M'$.
All entries of $M$ away from $(v,v)$ are plainly the corresponding entries
of $\widetilde M_B+\widetilde M'$.  At the shared diagonal entry,
$p_B\perp p'$ and \eqref{eq:leaf-gram} give
\[
  M_{vv}=\lVert z_v\rVert^2
  =\lVert p_B+p'\rVert^2
  =\lVert p_B\rVert^2+\lVert p'\rVert^2
  =(M_B)_{vv}+(M')_{vv}.
\]
Therefore
\begin{equation}\label{eq:matrix-leaf-split}
  M=\widetilde M_B+\widetilde M'.
\end{equation}

Lemma~\ref{lem:one-block} gives
\[
  M_B=\alpha_B^2J_{|V(B)|}.
\]
The graph $H'$ has fewer blocks than $H$, so the minimality of $(H,M)$ gives
\[
  M'=\sum_{C\in\calB(H')}
  \alpha_C^2\one_C\one_C^{\top}
\]
on $V(H')$.  Substituting these two identities into
\eqref{eq:matrix-leaf-split} yields
\[
  M=\sum_{C\in\calB(H)}
  \alpha_C^2\one_C\one_C^{\top},
\]
contrary to the choice of $(H,M)$.  This proves the proposition.
\end{pf}

Applying Proposition~\ref{prop:exact-decomposition} to the normalized
equality pair $(G,M^c)$ in \eqref{eq:endpoint-folding} gives the precise
spectral identity needed in the next section.

\begin{corollary}\label{cor:folded-identity}
Assume either endpoint equality in Lemma~\ref{lem:endpoint}.  Let $X$ be
defined by \eqref{eq:X-epsilon}, set $M=X\circ X$ as in
\eqref{eq:endpoint-M}, and let $M^c$ be its folding.  Then
\begin{equation}\label{eq:folded-block-identity}
  M^c
  =X\circ X+
  \sum_{\{u,v\}\in\binom{V(G)}2\setminus E(G)}
  X_{uv}^2(e_u-e_v)(e_u-e_v)^{\top}
  =\sum_{B\in\calB(G)}
  \alpha_B^2\one_B\one_B^{\top}.
\end{equation}
If $h\in\R^{V(G)}$ satisfies
\begin{equation}\label{eq:block-zero-sum}
  \sum_{u\in V(B)}h_u=0
  \qquad\text{for every }B\in\calB(G),
\end{equation}
then
\begin{equation}\label{eq:hadamard-kernel}
  (X\circ X)h=0.
\end{equation}
Moreover, for every non-edge $uv$,
\begin{equation}\label{eq:nonedge-equality}
  X_{uv}\ne0
  \quad\Longrightarrow\quad
  h_u=h_v
  \quad\text{for every }h\text{ satisfying \eqref{eq:block-zero-sum}}.
\end{equation}
\end{corollary}

\begin{pf}
The identity \eqref{eq:folded-block-identity} follows immediately from
Proposition~\ref{prop:exact-decomposition}.  Condition
\eqref{eq:block-zero-sum} implies $M^c h=0$ by the right-hand side of
\eqref{eq:folded-block-identity}.  Hence
\begin{align*}
  0=h^{\top}M^c h
  &=h^{\top}(X\circ X)h\\
  &\quad+
  \sum_{\{u,v\}\in\binom{V(G)}2\setminus E(G)}
  X_{uv}^2(h_u-h_v)^2.
\end{align*}
Both terms are nonnegative: $X\circ X\succeq0$ by the Schur product theorem,
and every summand in the second term is nonnegative.  Thus
$h^{\top}(X\circ X)h=0$.  For a positive semidefinite matrix $P$,
\[
  h^{\top}Ph=\lVert P^{1/2}h\rVert^2,
\]
so $h^{\top}Ph=0$ implies $P^{1/2}h=0$ and hence $Ph=0$.  Taking
$P=X\circ X$ proves \eqref{eq:hadamard-kernel}.  Every summand in the non-edge sum also
vanishes, which proves \eqref{eq:nonedge-equality}.
\end{pf}

\section{The bridge--clique obstruction}
\label{sec:obstruction}

We need a simple extension property for the blockwise zero-sum conditions.

\begin{lemma}\label{lem:extension}
Fix a block $B_0$ of a connected graph $G$.  Every assignment on $V(B_0)$
whose coordinates sum to zero extends to a vector $h\in\R^{V(G)}$
satisfying \eqref{eq:block-zero-sum}.

Moreover, root the block-cut tree at the block node $B_0$.  If a branch is
entered through a cut vertex $c$ with $h_c=0$, then the extension may be
chosen identically zero on every graph vertex in that branch.
\end{lemma}

\begin{pf}
Root the block-cut tree at $B_0$.  Suppose values have been assigned through
a block, and let $C$ be a child block entered through its parent cut vertex
$c$.  Choose a vertex $r\in V(C)\setminus\{c\}$, set
\[
  h_r=-h_c,
\]
and set all other as-yet unassigned coordinates of $C$ equal to zero.  The
coordinates on $C$ then sum to zero.  Continue away from the root.  The
assignment is consistent because every non-root block has a unique parent
cut vertex, and a cut vertex retains the value assigned when it first
appears.  If $h_c=0$, choose all new values in $C$ to be zero and repeat this
choice throughout every descendant block in that branch.
\end{pf}

\begin{proposition}[Bridge--clique obstruction]
\label{prop:bridge-clique}
Assume either endpoint equality in Lemma~\ref{lem:endpoint}.  No bridge of
$G$ can meet a nontrivial block.
\end{proposition}

\begin{pf}
By Proposition~\ref{prop:endpoint-blocks}, a nontrivial block is $B=K_t$ for
some $t\ge3$.  Suppose that a bridge $uv$ meets $B$ at $v$, and choose
$w\in V(B)\setminus\{v\}$.  With $X,\eps$ as in
\eqref{eq:X-epsilon}, \eqref{eq:X-on-edge} gives
\begin{equation}\label{eq:two-edge-values}
  X_{uv}=\frac{\eps}{2},
  \qquad
  X_{vw}=\eps\frac{t-1}{t}.
\end{equation}

We first claim that
\begin{equation}\label{eq:first-vanishing}
  X_{zw}=0
  \qquad(z\in N_G(u)\setminus\{v\}).
\end{equation}
Deleting the bridge $uv$ separates $u$ and $w$.  Every neighbor
$z\ne v$ of $u$ lies in the $u$-component of $G-uv$, while $w$ lies in the
$v$-component; hence $zw$ is a non-edge.  Choose
$r\in V(B)\setminus\{v,w\}$ and prescribe on $B$
\[
  h_v=0,
  \qquad h_w=1,
  \qquad h_r=-1,
\]
with all other coordinates on $B$ equal to zero.  The zero-sum condition on
the bridge block $\{u,v\}$ gives $h_u=-h_v=0$.  By
Lemma~\ref{lem:extension}, extend $h$ identically zero through the branch on
the $u$-side of the bridge.  In particular, $h_z=0\ne1=h_w$.
The contrapositive of \eqref{eq:nonedge-equality} gives
\eqref{eq:first-vanishing}.

We next claim that
\begin{equation}\label{eq:second-vanishing}
  X_{uz}=0
  \qquad(z\in N_G(w)\setminus\{v\}).
\end{equation}
The pair $uz$ is a non-edge.  Indeed, if $z=u$, then the edges $uv$, $vw$, and $wu$ form a triangle
containing $uv$;
if $z\ne u$ and $uz$ were an edge, then $u,z,w,v,u$ would be a cycle.
Either possibility contradicts that $uv$ is a bridge.

First suppose $z\in V(B)$.  Prescribe on $B$
\[
  h_v=1,
  \qquad h_z=0,
  \qquad h_r=-1
\]
for a vertex $r\in V(B)\setminus\{v,z\}$, and set all remaining coordinates
on $B$ equal to zero.  Extend this assignment by
Lemma~\ref{lem:extension}.  The bridge block then gives
$h_u=-h_v=-1\ne h_z$.

Now suppose $z\notin V(B)$.  Prescribe on $B$
\[
  h_v=1,
  \qquad h_w=0,
  \qquad h_r=-1
\]
for $r\in V(B)\setminus\{v,w\}$, with all other coordinates on $B$ equal
to zero.  Again the bridge block gives $h_u=-1$.  Since the edge $wz$ lies
in a block other than $B$, the vertex $w$ is a cut vertex and $z$ belongs
to a block-cut branch emanating from $B$ through $w$.  Because $h_w=0$,
Lemma~\ref{lem:extension} allows that entire branch to be assigned zero; in
particular, $h_z=0\ne h_u$.  In both cases, the contrapositive of
\eqref{eq:nonedge-equality} proves \eqref{eq:second-vanishing}.

The matrix $X$ is a spectral function of $A$, and therefore $AX=XA$.  At the
entry $(u,w)$, \eqref{eq:two-edge-values} and
\eqref{eq:first-vanishing} give
\[
  (AX)_{uw}
  =\sum_{z\sim u}X_{zw}
  =X_{vw}
  =\eps\frac{t-1}{t}.
\]
On the other hand, \eqref{eq:two-edge-values} and
\eqref{eq:second-vanishing} give
\[
  (XA)_{uw}
  =\sum_{z\sim w}X_{uz}
  =X_{uv}
  =\frac{\eps}{2}.
\]
These values are different because $t\ge3$, a contradiction.
\end{pf}

\section{Proof of the equality characterization}
\label{sec:proof-main}

\begin{pf}[Proof of Theorem~\ref{thm:main}]
The graph $K_1$ satisfies both assertions, so assume $n\ge2$.

If $G$ is a tree, then $G$ is bipartite, its adjacency spectrum is symmetric
about zero, and $|E(G)|=n-1$.  By \eqref{eq:trace-sum},
\[
  s^+(G)=s^-(G)=n-1.
\]
The spectrum of $K_n$ is $n-1,-1,\ldots,-1$, so
\[
  s^-(K_n)=n-1.
\]
This proves the easy directions.

Conversely, suppose first that $s^+(G)=n-1$ and that $G$ is not a tree.
Then $n\ge3$.  If $G$ had no cut vertex, Lemma~\ref{lem:no-cut-endpoint}
would force $G=K_n$, but $s^+(K_n)=(n-1)^2>n-1$, a contradiction.  Hence
$G$ has a cut vertex.  By Proposition~\ref{prop:endpoint-blocks}, every
block of $G$ is complete, so the block-cut tree is nontrivial.  Since $G$
is not a tree, it contains a nontrivial block.  Lemma~\ref{lem:positive-leaf}
shows that every leaf block is a bridge.

In a nontrivial block-cut tree, every cut-vertex node has degree at least
two, so every leaf node is a block node.  Choose a path from a nontrivial block
node to a leaf block node, and list its block nodes in order as
$B_0,B_1,\ldots,B_k$.  The
terminal block $B_k$ is a bridge.  Let $j$ be the least index for which
$B_j$ is a bridge.  Then $j\ge1$, and $B_{j-1}$ is non-bridge; since every
block is complete, $B_{j-1}$ is a nontrivial complete block.  The blocks
$B_{j-1}$ and $B_j$ meet at the intervening cut vertex, contrary to
Proposition~\ref{prop:bridge-clique}.  Therefore $G$ is a tree.

Now suppose that $s^-(G)=n-1$.  If $G$ is neither a tree nor complete,
Lemma~\ref{lem:no-cut-endpoint} shows that $G$ has a cut vertex.  Every
block is complete by Proposition~\ref{prop:endpoint-blocks}.  Since $G$ is
not complete, Lemma~\ref{lem:negative-leaf} shows that every leaf block is a
bridge.  The graph is not a tree, so it has a nontrivial block.  A path from a nontrivial block
to a leaf block again has a first bridge block whose preceding block is
nontrivial, contradicting Proposition~\ref{prop:bridge-clique}.  Hence $G$ is a
tree or a complete graph.
\end{pf}

\begin{pf}[Proof of Corollary~\ref{cor:strict}]
Liu, Tang, and Zhang proved the non-strict inequalities
\[
  \min \{s^+(G),s^-(G)\}\ge |V(G)|-1
\]
for every connected graph \cite{LTZ}.  If $G$ is neither a tree nor a
complete graph, Theorem~\ref{thm:main} excludes equality in both
inequalities.
\end{pf}

\section*{Acknowledgments}

Supported by National Natural Science Foundation of China (12571360, 12331012), Excellent 
University Research and Innovation Team in Anhui Province (2024AH010002, 2025AHGXZK10041).

The core ideas and proof strategy presented in this paper
were conceived and developed by the author. Eureka was used as an auxiliary
tool to test and check some of these ideas, after which the author
independently examined and verified the mathematical arguments. Eureka is a
multi-agent system developed by JIUCHONG at the University of Science and
Technology of China for mathematical research through human--AI interaction.
The author assumes full responsibility for the paper's content.

\end{document}